\documentclass[11pt]{amsart}
\usepackage{amssymb,amsthm,amsmath,epsfig,latexsym}
\usepackage{calc,times,verbatim}
\usepackage{geometry}
\usepackage{color}

\begin{document}

\newcommand{\mmbox}[1]{\mbox{${#1}$}}
\newcommand{\proj}[1]{\mmbox{{\mathbb P}^{#1}}}
\newcommand{\Cr}{C^r(\Delta)}
\newcommand{\CR}{C^r(\hat\Delta)}
\newcommand{\affine}[1]{\mmbox{{\mathbb A}^{#1}}}
\newcommand{\Ann}[1]{\mmbox{{\rm Ann}({#1})}}
\newcommand{\caps}[3]{\mmbox{{#1}_{#2} \cap \ldots \cap {#1}_{#3}}}
\newcommand{\Proj}{{\mathbb P}}
\newcommand{\N}{{\mathbb N}}
\newcommand{\Z}{{\mathbb Z}}
\newcommand{\R}{{\mathbb R}}
\newcommand{\A}{{\mathcal{A}}}
\newcommand{\Tor}{\mathop{\rm Tor}\nolimits}
\newcommand{\Ext}{\mathop{\rm Ext}\nolimits}
\newcommand{\Hom}{\mathop{\rm Hom}\nolimits}
\newcommand{\im}{\mathop{\rm Im}\nolimits}
\newcommand{\rank}{\mathop{\rm rank}\nolimits}
\newcommand{\supp}{\mathop{\rm supp}\nolimits}
\newcommand{\arrow}[1]{\stackrel{#1}{\longrightarrow}}
\newcommand{\CB}{Cayley-Bacharach}
\newcommand{\coker}{\mathop{\rm coker}\nolimits}
\newcommand{\fm}{\mathfrak m}
\newcommand{\rees}{{R[It]}}
\newcommand{\fiber}{{\mathcal{F}}(I)}
\sloppy
\theoremstyle{plain}

\newtheorem*{thm*}{Theorem}
\newtheorem{defn0}{Definition}[section]
\newtheorem{prop0}[defn0]{Proposition}
\newtheorem{quest0}[defn0]{Question}
\newtheorem{thm0}[defn0]{Theorem}
\newtheorem{lem0}[defn0]{Lemma}
\newtheorem{corollary0}[defn0]{Corollary}
\newtheorem{example0}[defn0]{Example}
\newtheorem{remark0}[defn0]{Remark}
\newtheorem{conj0}[defn0]{Conjecture}

\newenvironment{defn}{\begin{defn0}}{\end{defn0}}
\newenvironment{conj}{\begin{conj0}}{\end{conj0}}
\newenvironment{prop}{\begin{prop0}}{\end{prop0}}
\newenvironment{quest}{\begin{quest0}}{\end{quest0}}
\newenvironment{thm}{\begin{thm0}}{\end{thm0}}
\newenvironment{lem}{\begin{lem0}}{\end{lem0}}
\newenvironment{cor}{\begin{corollary0}}{\end{corollary0}}
\newenvironment{exm}{\begin{example0}\rm}{\end{example0}}
\newenvironment{rem}{\begin{remark0}\rm}{\end{remark0}}

\newcommand{\defref}[1]{Definition~\ref{#1}}
\newcommand{\conjref}[1]{Conjecture~\ref{#1}}
\newcommand{\propref}[1]{Proposition~\ref{#1}}
\newcommand{\thmref}[1]{Theorem~\ref{#1}}
\newcommand{\lemref}[1]{Lemma~\ref{#1}}
\newcommand{\corref}[1]{Corollary~\ref{#1}}
\newcommand{\exref}[1]{Example~\ref{#1}}
\newcommand{\secref}[1]{Section~\ref{#1}}
\newcommand{\remref}[1]{Remark~\ref{#1}}
\newcommand{\questref}[1]{Question~\ref{#1}}

\newcommand{\std}{Gr\"{o}bner}
\newcommand{\jq}{J_{Q}}

\title{Ideals generated by $a$-fold products of linear forms have linear graded free resolution}
\author{Ricardo Burity, \c{S}tefan O. Toh\v{a}neanu and Yu Xie}

\keywords{linear free resolution, ideals generated by products of linear forms,  Orlik-Terao algebra, ideal of fiber type, star configurations \\
\indent The first author was partially supported by CAPES - Brazil (grant: PVEX - 88881.336678/2019-01). \\
\indent Burity's address: Departamento de Matemática, Universidade Federal da Paraiba, J. Pessoa, Paraiba, 58051-900, Brazil,
Email: ricardo@mat.ufpb.br.\\
\indent Tohaneanu's address: Department of Mathematics, University of Idaho, Moscow, Idaho 83844-1103, USA, Email: tohaneanu@uidaho.edu.\\
\indent Xie's address: Department of Mathematics, Widener University, Chester, Pennsylvania 19013, USA, Email: yxie@widener.edu.}

\begin{abstract}
Given $\Sigma\subset R:=\mathbb K[x_1,\ldots,x_k]$, where $\mathbb K$ is a field of characteristic 0, any finite collection of linear forms, some possibly proportional, and any $1\leq a\leq |\Sigma|$, we prove that $I_a(\Sigma)$, the ideal generated by all $a$-fold products of $\Sigma$, has linear graded free resolution. This allows us to determine a generating set for the defining ideal of the Orlik-Terao algebra of the second order of a line arrangement in $\mathbb P_{\mathbb{K}}^2$, and to conclude that for the case $k=3$, and $\Sigma$ defining such a line arrangement, the ideal $I_{|\Sigma|-2}(\Sigma)$ is of fiber type. We also prove several conjectures of symbolic powers for defining ideals of star configurations of any codimension $c$.
\end{abstract}

\maketitle

\section{Introduction}

Let $R:=\mathbb K[x_1,\ldots,x_k]$ be the ring of (homogeneous) polynomials with coefficients in a field $\mathbb K$, with the standard grading. Denote ${\frak m}:=\langle x_1,\ldots,x_k\rangle$ to be the irrelevant maximal ideal of $R$. Let $\ell_1,\ldots,\ell_n$ be linear forms in $R$, some possibly proportional, and denote this collection by $\Sigma=(\ell_1,\ldots,\ell_n)\subset R$. For $\ell\in \Sigma$, by $\Sigma\setminus\{\ell\}$ we will understand the collection of linear forms of $\Sigma$ from which one copy of $\ell$ has been removed. Also, we denote $|\Sigma|=n$, and ${\rm rk}(\Sigma):={\rm ht}(\langle \ell_1,\ldots,\ell_n\rangle)$.

Let $1\leq a\leq n$ be an integer and define {\em the ideal generated by (all) $a$-fold products of $\Sigma$} to be the ideal of $R$ $$I_a(\Sigma):=\langle \{\ell_{i_1}\cdots \ell_{i_a}|1\leq i_1<\cdots<i_a\leq n\}\rangle.$$ We also make the convention $I_0(\Sigma):=R$, and $I_b(\Sigma)=0$, for all $b>n$. Also, if $\Sigma=\emptyset$, $I_a(\Sigma)=0$, for any $a\geq 1$. In some places along the exposition we will find it more convenient to use the notation $I_a(\ell_1\cdots\ell_n)$ for the same ideal $I_a(\Sigma)$. Also, an element $\ell_{i_1}\cdots \ell_{i_a}$ will be called standard generator of $I_a(\Sigma)$.

\medskip

A homogeneous ideal $I\subset R$ generated in degree $d$ is said to have {\em linear (minimal) graded free resolution}, if one has the graded free resolution $$0\rightarrow R^{n_{b+1}}(-(d+b))\rightarrow\cdots\rightarrow R^{n_2}(-(d+1))\rightarrow R^{n_1}(-d)\rightarrow R \rightarrow R/I\rightarrow 0,$$ for some positive integer $b$. The integers $n_j\geq 1$ are called the {\em Betti numbers} of $R/I$. By convention, the zero ideal has linear graded free resolution. Also we say that $R/I$ has linear graded free resolution if and only if $I$ has linear graded free resolution.

\cite[Conjecture 1]{AnGaTo} states that for any collection of linear forms $\Sigma$, and any $1\leq a\leq|\Sigma|$, the ideals $I_a(\Sigma)$ (or $R/I_a(\Sigma)$) have linear graded free resolution. In \cite{ToXi} it is presented the current state of this conjecture, as well as it is shown that the conjecture is true whenever the support of $\Sigma$ (i.e., the set of nonproportional elements of $\Sigma$) is generic.

In Section 2,  we show the validity of this conjecture in its full generality (see Theorem \ref{theorem1}). This result lines up with several other results in literature where various modules/ideals are shown to have graded linear free resolutions: first, one needs to mention the landmark paper \cite{EiGo} from which a lot of the techniques are derived from their general exposition, and second, the landmark paper \cite[Theorem 3.1]{CoHe} where it is shown that any product of ideals generated by linear forms has linear free resolution.

The proof of the similar results in \cite{ToXi} or \cite{CoHe} relies heavily on having the knowledge of the primary decomposition of the saturation of the corresponding ideals (see \cite[Proposition 2.2]{ToXi}, and \cite[Lemma 3.2]{CoHe}), because then via Remark \ref{remark0} below, and an induction argument that uses the inequalities of regularity under short exact sequences, one would obtain the desired results. Our approach is not that much different, except that we prove the claimed primary decomposition simultaneously with the linear free resolution result, using the same induction for both. The proof of the similar result \cite[Theorem 2.4]{To3} uses a completely different strategy that allows to find also the betti numbers.

\begin{rem}\label{remark0} Let $J\subset R$ be an ideal generated in degree $a$. Then $J\subseteq J^{\rm sat}\cap {\frak m}^a$.\footnote{If $I\subset R$ is a homogeneous ideal, by definition $I^{\rm sat}:=\{f\in R \, |\, \exists \, n(f)\geq 0 \mbox{ such that }{\frak m}^{n(f)}\cdot f\subset I\}$.} If $R/J$ has linear graded free resolution (equivalently, ${\rm reg}(R/J)=a-1$), since ${\rm H}_{\frak m}^0(R/J)=J^{\rm sat}/J$, by \cite[Theorem 4.3]{Ei}, we have $(J^{\rm sat}/J)_e=0, \mbox{ for any }e\geq a$. This means that $J^{\rm sat}\cap {\frak m}^a\subseteq J$, and therefore $$J=J^{\rm sat}\cap {\frak m}^a.$$
\end{rem}

\medskip

In Section 3, we apply Theorem \ref{theorem1} to study Rees algebras of $I_a(\Sigma)$. Based on this theorem,  we observe that the ideal $I_a(\Sigma)$ has linear powers for any $\Sigma$ and any $a$ (see Remark \ref{gensRees}). This property may give more advantages in studying the Rees algebras of $I_a(\Sigma)$. A first question is if an ideal with linear powers is of fiber type, and \cite[Example 2.6]{BrCoVa} gives an example of an ideal with linear powers that is not of fiber type. The ideal presented in the counter example is generated by products of linear forms, but it is not an ideal of the type $I_a(\Sigma)$. Therefore we are conjecturing that the ideals generated by $a$-fold products of linear forms are of fiber type. This conjecture was verified to be true by  \cite[Theorem 4.2]{GaSiTo}  when $\Sigma$ defines a hyperplane arrangement, and $a=|\Sigma|-1$.  The proof didn't appeal to the linear powers property, since Theorem \ref{theorem1} was not available at that time. Here in this paper, we prove this conjecture for another case, i.e., when $\Sigma$ defines a line arrangement $\A$ in $\mathbb P^2$ with $a=|\A|-2$. The high level of technical computations are forcing us to restrict to this situation, but a vague idea on how to deal with the general case of the conjecture seems visible. Another, more conceptual, reason why one would look at $a=|\A|-2$ and 3 variables, comes from hyperplane arrangement community where there is an interest into looking at special fibers of ideals generated by products of linear forms indexed by the independent sets of the matroid of the hyperplane arrangement; in $\mathbb P^2$, any two lines of a line arrangement are independent. So, in Theorem \ref{fibertype} we determine a set of generators for the defining ideal of the special fiber of $I$ (also known as the Orlik-Terao algebra of the second order, see \cite{To3}), and in Theorem \ref{fiber-type} we show that $I$ is of fiber type. Both these results answer affirmatively the two related conjectures stated in \cite{To3}, for the case $k=3$.

In the last section, we study $\mathcal{A}=\{H_1,\ldots,H_s\}$  a collection of $s\geq N+1$ distinct hyperplanes in $\mathbb P^N$ (so $N=k-1$). We assume these hyperplanes meet {\it properly}, by which we mean that the intersection of any $j$ of these hyperplanes is either empty or has codimension $j$. For any $1\leq c\leq N$, the star configuration of codimension~$c$ is defined as the union of the codimension $c$ linear varieties defined by all the intersections of these hyperplanes, taken $c$ at a time. We prove several conjectures of symbolic powers for defining ideals of such configurations.  These results are  applications of \cite[Theorem 2.3 and Theorem 3.2]{ToXi}, but since we didn't present them there, and because our current Theorem \ref{theorem1} generalizes \cite[Theorem 2.3 and Theorem 3.2]{ToXi}, we are including them here. It will be interesting to discuss various results and conjectures concerning symbolic powers of ideals that generalize defining ideals of star configurations. Also, these ideals need to be saturated. One instance when this happens is the case of ideals generated by $(s-N+1)$-fold products of linear forms, with the condition that any $N$ of the linear forms in $\Sigma$ are linearly independent (see \cite[Proposition 2.1]{To2}); if any $N+1$ linear forms are linearly independent the same ideal defines a zero-dimensional star configuration in $\mathbb P^{N}$. Our Theorem \ref{theorem1} will be helpful to show parts of the proofs of such results, but to complete such arguments one needs to have knowledge of the $\alpha$-invariant of the symbolic powers: for  star configurations, \cite[Lemma 8.4.7]{BRHKKSS09}, \cite[Lemma 2.4.1]{BoHa}, and \cite[Corollary 4.6]{GeHaMi},  and for $N=2$, \cite[Proposition 3.2]{ToXi2} present very useful lower bounds. In future work we will tackle more general cases.

\bigskip

\section{Ideals generated by $a$-fold products of linear forms}

\medskip

Let $\Sigma=(\underbrace{\ell_1,\ldots,\ell_1}_{m_1},\ldots,\underbrace{\ell_s,\ldots,\ell_s}_{m_s})$ be a collection of linear forms in $R:=\mathbb K[x_1,\ldots,x_k]$, with $\gcd(\ell_i,\ell_j)=1$ if $i\neq j$. The {\em support of $\Sigma$} is ${\rm Supp}(\Sigma):=\{\ell_1,\ldots,\ell_s\}$. Assume ${\rm rk}(\Sigma)=k$. Let $n:=m_1+\cdots+m_s$.

Let $1\leq a\leq n$, and consider $I_a(\Sigma)$, or $I_a(\ell_1^{m_1}\cdots\ell_s^{m_s})$, the ideal generated by (all) $a$-fold products of linear forms.

\medskip

Next are a couple of observations that will help with understanding the notations used, and how we look at prime ideals containing $I_a(\Sigma)$.

\begin{lem}\label{primes} Let $P\subset R$ be a prime ideal. Then $I_a(\Sigma)\subseteq P$ if and only if there are at least $n-a+1$ elements of $\Sigma$ (counted with multiplicity) that belong also to $P$.
\end{lem}
\begin{proof} This argument is presented in \cite{To}, but we repeat it here for complete exposition. For the purpose of this proof we go back to the notation at the beginning: $\Sigma=(\ell_1,\ldots,\ell_n)$.

Suppose $I_a(\Sigma)\subseteq P$. Then, the standard generator $\ell_1\ell_2\cdots\ell_a\in P$. So one of its factors must belong to $P$; say $\ell_1\in P$. Next we look at the standard generator $\ell_2\ell_3\cdots\ell_{a+1}\in P$. One of its factors must belong to $P$; say $\ell_2\in P$ (note that we don't exclude the possibility that $\ell_1$ and $\ell_2$ may be proportional). And so forth, we find $n-a+1$ linear forms from $\Sigma$ that belong to $P$.

For the converse, if $P$ has $n-a+1$ elements of $\Sigma$, every standard generator must have one of these linear forms as a factor, and hence it belongs to $P$, giving that $I_a(\Sigma)\subseteq P$.
\end{proof}

\medskip

For any prime $P\subset R$, define {\em the closure of $P$ in $\Sigma$} to be the subcollection of linear forms of $\Sigma$ (taken accordingly with their multiplicity) that belong also to $P$:

$${\rm cl}_{\Sigma}(P):=\{\ell\in\Sigma|\ell\in P\}.$$ Also define $\nu_{\Sigma}(P):=|{\rm cl}_{\Sigma}(P)|$.

We want to look at prime ideals generated by subsets of linear forms from $\Sigma$; denote the set of all such linear primes by $\Gamma(\Sigma)$.

Let $\frak q\in \Gamma(\Sigma)$. Then $\nu_{\Sigma}(\frak q)$ is the number of linear forms from $\Sigma$, counted with multiplicity, that belong to $\frak q$. For example, the irrelevant maximal ideal $\frak m:=\langle x_1,\ldots,x_k\rangle$ is the only element of $\Gamma(\Sigma)$ of height equal to $k$, because ${\rm rank}(\Sigma)=k$; hence $\frak m=\langle \ell_1,\ldots,\ell_s\rangle$ and $\nu_{\Sigma}(\frak m)=m_1+\cdots+m_s=n$.

If $\frak q\in \Gamma(\Sigma)$ with ${\rm ht}(\frak q)=c\leq k-1$, then there exists $c$ linearly independent linear forms of ${\rm Supp}(\Sigma)$, say $\ell_{i_1},\ldots,\ell_{i_c}$, such that $\frak q=\langle \ell_{i_1},\ldots,\ell_{i_c}\rangle$, with possibly $${\rm cl}_{\Sigma}(\frak q)= (\underbrace{\ell_{i_1},\ldots,\ell_{i_1}}_{m_{i_1}},\ldots,\underbrace{\ell_{i_c},\ldots,\ell_{i_c}}_{m_{i_c}}, \underbrace{\ell_{i_{c+1}},\ldots,\ell_{i_{c+1}}}_{m_{i_{c+1}}}, \ldots,\underbrace{\ell_{i_u},\ldots,\ell_{i_u}}_{m_{i_u}}).$$ So $\nu_{\Sigma}(\frak q)=m_{i_1}+\cdots+m_{i_u}$. Note that in \cite{ToXi}, since ${\rm Supp}(\Sigma)$ is generic, $u=c$.

Since $n=m_1+\cdots+m_s$, following the notations in \cite{ToXi}, $\mu(i_1,\ldots,i_u)=a-(n-\nu_{\Sigma}(\frak q))$, which is maximal possible, since we cannot have more than $\nu_{\Sigma}(\frak q)$ linear forms that generate $\frak q$.

By \cite[Lemma 2.1]{ToXi}, we have that
$$
I_a(\Sigma)\subseteq \bigcap_{\frak p\in \Gamma(\Sigma)}\frak p^{a-n+\nu_{\Sigma}(\frak p)},
$$
where  $\mathfrak{p}^{a-n+\nu_{\Sigma}(\mathfrak{p})}$ is replaced by $R$ if the power $a-n+\nu_{\Sigma}(\mathfrak{p})\leq 0$. By Lemma \ref{primes}, observe $a-n+\nu_{\Sigma}(\frak p)\leq 0$ if and only if $\frak p \nsupseteq I_a(\Sigma)$.

\medskip

\begin{thm} \label{theorem1} Let $\Sigma=(\underbrace{\ell_1,\ldots,\ell_1}_{m_1},\ldots,\underbrace{\ell_s,\ldots,\ell_s}_{m_s})$ be a collection of linear forms in $R:=\mathbb K[x_1,\ldots,x_k]$, with $\gcd(\ell_i,\ell_j)=1$ if $i\neq j$. Denote $n=|\Sigma|$ and $1\leq a\leq n$. Then the ideal $I_a(\Sigma)$
has linear graded free resolution and the following primary decomposition
$$
I_a(\Sigma)= \bigcap_{\frak p\in \Gamma(\Sigma)}\frak p^{a-n+\nu_{\Sigma}(\frak p)}.
$$
\end{thm}

\begin{proof} After a change of variables, and possibly embedding in a smaller ring, we may suppose ${\rm rk}(\Sigma)=k$.

We will prove the result by induction on the pairs $(|\Sigma|,{\rm rk}(\Sigma))$, with $|\Sigma|\geq {\rm rk}(\Sigma)\geq 2$.

\medskip

{\bf Base cases.} If ${\rm rk}(\Sigma)=2$, then, modulo an embedding, \cite[Theorem 2.2]{To3} gives the linear free resolution part. Then, via Remark \ref{remark0}, one must show that $$I_a(\Sigma)^{\rm sat}=\bigcap_{i=1}^s\langle l_i\rangle^{a-n+m_i}.$$ But this is true from \cite[Proposition 2.3]{AnGaTo}.

If $|\Sigma|={\rm rk}(\Sigma)$, then after a change of variables we can assume that $\Sigma=(x_1,\ldots,x_k)$. This is the Boolean arrangement, which leads to a particular case of star configurations, and therefore $I_a(\Sigma)$ has graded linear free resolution and the primary decomposition for any $1\leq a\leq |\Sigma|$ (see part (3) in the Introduction of \cite{To3} or see \cite{ToXi}).

\medskip

{\bf Inductive step.} Suppose $|\Sigma|>{\rm rk}(\Sigma)\geq 3$.

Let $\ell\in\Sigma$ and $\Sigma':=\Sigma\setminus\{\ell\}$. Denote $n':=n-1=|\Sigma'|$. We will prove the following claim:

\bigskip

\noindent CLAIM: \,  $I_a(\Sigma):\ell=I_{a-1}(\Sigma')$.

\begin{proof} By  \cite[Lemma 2.1]{ToXi}, we have that
$$
I_a(\Sigma)\subseteq \bigcap_{\frak p\in \Gamma(\Sigma)}\frak p^{a-n+\nu_{\Sigma}(\frak p)}\subseteq \bigcap_{\frak p\in \Gamma(\Sigma^{\prime})}\frak p^{a-n+\nu_{\Sigma}(\frak p)}.
$$ The last inclusion comes from $\Gamma(\Sigma)\subseteq \Gamma(\Sigma')$, and therefore in the last term we intersect fewer elements.

Hence
$$
I_a(\Sigma):\ell\subseteq \bigcap_{\frak p\in \Gamma(\Sigma')}(\frak p^{a-n+\nu_{\Sigma}(\frak p)}:\ell).
$$

Now, let $\frak p\in\Gamma(\Sigma')$. Then, since $\frak p$ is a linear prime, for any power $u$, we have

$$\frak p^u:\ell = \left\{
  \begin{array}{ll}
    \frak p^u, & \hbox{if } \ell\notin\frak p, \\
    \frak p^{u-1}, & \hbox{if } \ell\in\frak p.
  \end{array}
\right.$$

\begin{itemize}
  \item[(1)] If $\ell\notin\frak p$, then $\nu_{\Sigma'}(\frak p)=\nu_{\Sigma}(\frak p)$. So $$\frak p^{a-n+\nu_{\Sigma}(\frak p)}:\ell= \frak p^{a-n+\nu_{\Sigma}(\frak p)}=\frak p^{(a-1)-n'+\nu_{\Sigma'}(\frak p)}.$$
  \item[(2)] If $\ell\in\frak p$, then $\nu_{\Sigma'}(\frak p)=\nu_{\Sigma}(\frak p)-1$, so
$$\frak p^{a-n+\nu_{\Sigma}(\frak p)}:\ell=\frak p^{a-n+\nu_{\Sigma}(\frak p)-1}=\frak p^{(a-1)-n'+\nu_{\Sigma'}(\frak p)}.$$
\end{itemize}

\medskip

Hence we have
$$
\bigcap_{\frak p\in \Gamma(\Sigma^{\prime})}(\frak p^{a-n+\nu_{\Sigma}(\frak p)}:\ell)=\bigcap_{\frak p\in \Gamma(\Sigma^{\prime})}\frak p^{(a-1)-n^{\prime}+\nu_{\Sigma^{\prime}}(\frak p)}.
$$

\medskip

By inductive hypotheses,  $I_{a-1}(\Sigma')$ has the following primary decomposition
$$
I_{a-1}(\Sigma')=\bigcap_{\frak p\in\Gamma(\Sigma')}\frak p^{(a-1)-n'+\nu_{\Sigma'}(\frak p)}.
$$

Everything put together gives $I_a(\Sigma):\ell\subseteq I_{a-1}(\Sigma')$. Since the reverse inclusion is obvious, we conclude the proof of the CLAIM.
\end{proof}

\bigskip

Now we have the short exact sequence of graded $R$-modules:

$$0\longrightarrow \frac{R(-1)}{I_{a-1}(\Sigma')}\longrightarrow \frac{R}{I_a(\Sigma)}\longrightarrow \frac{R}{\langle\ell, I_a(\Sigma)\rangle}\longrightarrow 0.$$

\medskip

By inductive hypotheses, $I_{a-1}(\Sigma')$ has linear graded free resolution and
$${\rm reg}\left(\frac{R(-1)}{I_{a-1}(\Sigma')}\right)=(a-2)+1=a-1.$$

\medskip

Next we deal with the rightmost nonzero module. We can suppose, after a change of variables, that $\ell=\ell_s=x_k$. Suppose that for all $1\leq i\leq s-1$, $\ell_i=\bar{\ell}_i+c_ix_k, c_i\in\mathbb K$, where $\bar{\ell}_i$ are linear forms in variables $x_1,\ldots,x_{k-1}$.

Consider
 $$\bar{\Sigma}:=(\underbrace{\bar{\ell}_1,\ldots,\bar{\ell}_1}_{m_1}, \ldots, \underbrace{\bar{\ell}_{s-1},\ldots,\bar{\ell}_{s-1}}_{m_{s-1}})\subset \bar{R}:=\mathbb K[x_1,\ldots,x_{k-1}].
 $$
 Since we assumed that ${\rm rk}(\Sigma)=k$, then ${\rm rk}(\bar{\Sigma})=k-1$. By inductive hypotheses, $I_a(\bar{\Sigma})$ has a linear graded free resolution and we have
 $${\rm reg}\left(\frac{\bar{R}}{I_a(\bar{\Sigma})}\right)=a-1.$$

\medskip

\noindent
But $R/\langle\ell,I_a(\Sigma)\rangle$ and $\bar{R}/I_a(\bar{\Sigma})$ are isomorphic as $R$-modules, so they have the same regularity (see \cite[Corollary 4.6]{Ei}).

\medskip

Applying the inequalities of regularity under short exact sequence (see \cite[Corollary 20.19 b.]{Ei2}), we can conclude that ${\rm reg}(R/I_a(\Sigma))\leq a-1$, and therefore that $I_a(\Sigma)$ has linear graded free resolution.

\bigskip

Now we are going to show
$$
I_a(\Sigma)= \bigcap_{\frak p\in \Gamma(\Sigma)}\frak p^{a-n+\nu_{\Sigma}(\frak p)}=\bigcap_{\frak p\in \Gamma(\Sigma)\setminus \{\frak m\}}\frak p^{a-n+\nu_{\Sigma}(\frak p)}\bigcap \frak m^a.
$$
Since $I_a(\Sigma)$ has linear graded free resolution, by Remark \ref{remark0}, we have $I_a(\Sigma)=I_a(\Sigma)^{\rm sat}\cap \frak m^a$.
So we just need to show
$$
I_a(\Sigma)^{\rm sat}=\bigcap_{\frak p\in \Gamma(\Sigma)\setminus \{\frak m\}}\frak p^{a-n+\nu_{\Sigma}(\frak p)}.
$$

Let $J=\bigcap_{\frak p\in\Gamma(\Sigma)\setminus \{\frak m\}}\frak p^{a-n+\nu_{\Sigma}(\frak p)}$.

To prove the above equality, we just need to show $I_a(\Sigma)$ and $J$ are the same after localizing at any prime ideal $Q$ of height $\leq k-1$.

Let $Q$ be such a prime ideal. Let $\frak q:=\langle {\rm cl}_{\Sigma}(Q)\rangle\in \Gamma(\Sigma)$. Obviously, $\ell\in \Sigma$ doesn't belong to $Q$ if and only if it doesn't belong to $\frak q$, hence $\nu_{\Sigma}(\frak q)=\nu_{\Sigma}(Q)$. Also, by Lemma \ref{primes}, $I_a(\Sigma)\subset Q$ if and only if $\nu_{\Sigma}(\frak q)\geq n-a+1$.

We also have ${\rm ht}(\frak q)=c\leq k-1$. We may assume $\frak q=\langle \ell_1, \ldots, \ell_c\rangle$. Let $\ell_{u+1}, \ldots \ell_s$ be all of the linear forms in ${\rm Supp}(\Sigma)$ that do not belong to $\frak{q}$ (and hence they do not belong to $Q$, meaning that they are invertible elements in $R_Q$). Then, by our discussions at the beginning of the section, $\nu_{\Sigma}(\frak q)=m_{1}+\cdots+m_{u}$.

Observe for any collection $(\ell_1,\ldots,\ell_s)$ of linear forms, and for any $1\leq a\leq n$, one has $$I_a(\Sigma)=I_a(\ell_1^{m_1}\cdots\ell_{s-1}^{m_{s-1}}\ell_s^{m_s})=\ell_s^{m_s}I_{a-m_s}(\ell_1^{m_1}\cdots\ell_{s-1}^{m_{s-1}})+ \cdots+\ell_s I_{a-1}(\ell_1^{m_1}\cdots\ell_{s-1}^{m_{s-1}})+I_a(\ell_1^{m_1}\cdots\ell_{s-1}^{m_{s-1}}).$$
Under localization at $Q$, $\ell_s$ is invertible, and since
$$I_{a-m_s}(\ell_1^{m_1}\cdots\ell_{s-1}^{m_{s-1}})\supset\cdots\supset I_a(\ell_1^{m_1}\cdots\ell_{s-1}^{m_{s-1}}),$$
we have $$I_a(\Sigma) R_Q=I_{a-m_s}(\ell_1^{m_1}\cdots\ell_{s-1}^{m_{s-1}}) R_Q.$$
Doing this for all $\ell_{u+1},\ldots,\ell_s$ we have
$$I_a(\Sigma) R_Q=I_{a-n+\nu_{\Sigma}(\frak q)}(\ell_1^{m_1}\cdots\ell_u^{m_u}) R_Q,$$
since $\nu_{\Sigma}(\frak q)=m_1+\cdots+m_u$.

Next we look at $$\Sigma':={\rm cl}_{\Sigma}(Q)=(\underbrace{\ell_1,\ldots,\ell_1}_{m_1},\ldots,\underbrace{\ell_u,\ldots,\ell_u}_{m_u})\subset \Sigma.$$ Since ${\rm rk} (\Sigma')={\rm ht}(\frak q)=c\leq k-1$, by induction hypothesis, we have

$$
I_{a-n+\nu_{\Sigma}(\frak q)}(\Sigma')= \bigcap_{\frak p \in \Gamma(\Sigma')}\frak p^{a-n+\nu_{\Sigma}(\frak q)-\nu_{\Sigma}(\frak q)+\nu_{\Sigma'}(\frak p)}=\bigcap_{\frak p\in \Gamma(\Sigma')}\frak p^{a-n+\nu_{\Sigma'}(\frak p)}.
$$

We have $$J R_Q = \left(\bigcap_{\frak p\in \Gamma(\Sigma)}\frak p^{a-n+\nu_{\Sigma}(\frak p)}\right) R_Q.$$ In this intersection, if $\frak p \nsubseteq Q$, then $\frak p R_Q = R_Q$, so we are interested only in $\frak p\subseteq Q$. But then, $\frak p\subseteq \frak q$, and hence $\frak p\in\Gamma(\Sigma')$ and $\nu_{\Sigma}(\frak p)=\nu_{\Sigma'}(\frak p)$.

Everything put together gives

$$
J R_Q = \left(\bigcap_{\frak p\in \Gamma(\Sigma')}\frak p^{a-n+\nu_{\Sigma'}(\frak p)}\right) R_Q=I_{a-n+\nu_{\Sigma}(\frak q)}(\Sigma')R_Q = I_a(\Sigma) R_Q.
$$
And we completed the proof. \end{proof}

\bigskip

\section{Rees algebras of line arrangements}

\medskip

In \cite[Section 3]{To3} there were several conjectures  concerning the generators of some elimination algebras, such as the special fiber, and now we can shed more light when the ambient ring is the polynomial ring in three variables.

\subsection{Basic definitions and concepts.} The \emph{Rees algebra} of an ideal $I$ in a ring $R$ is the homogeneous $R$-subalgebra of the standard graded polynomial $R[t]$ in one variable over $R$, generated by the elements $at, a\in I$, of degree $1$. Fixing a set of generators of $I$ determines a surjective homomorphism of $R$-algebras from a polynomial ring over $R$ to $R[It]$. The kernel of such a map is called a {\em presentation ideal} of $R[It]$. If $R=\mathbb K[x_1,\dots, x_k]$ is a standard graded polynomial ring over a field $\mathbb K$ and
$I=\langle g_1,\ldots, g_n\rangle$ is an ideal generated by forms $ g_1,\dots, g_n$ of the same degree, consider $T=R[y_1,\ldots,y_n]=\mathbb K[x_1,\dots, x_k; y_1,\ldots,y_n]$, a standard bigraded $\mathbb K$-algebra with $\deg x_i=(1,0)$ and $\deg y_j=(0,1)$.
Using the given generators to obtain an $R$-algebra homomorphism $$\varphi: T=R[y_1,\ldots,y_n]\longrightarrow R[It],\, y_i\mapsto g_it,$$
yields a presentation ideal $\mathcal{I}$ which is bihomogeneous in the bigrading of $T$.\footnote{Here we'll be talking about the presentation ideal of $R[It]$ by fixing a particular set of homogeneous generators of $I$ of the same degree.}

In this polynomial setup, one defines the \emph{special fiber}  $\fiber:=\rees\otimes_R R/{\fm}\simeq \oplus_{s\geq 0}I^s/{\fm}I^s$, where $\fm=\langle x_1,\ldots,x_k\rangle\subset R$. The Krull dimension of the special fiber $\ell(I):=\dim \fiber$ is called the {\em analytic spread} of $I$.

As noted before, the presentation ideal of $R[It]$
$$\mathcal I=\bigoplus_{(u,v)\in {\mathbb N} \times  {\mathbb N}} {\mathcal I}_{(u,v)},$$
is a bihomogeneous ideal in the standard bigrading of $T$. Two basic subideals of $\mathcal I$ are $\langle \mathcal I_{(0,-)}\rangle$ and $\langle \mathcal I_{(-,1)}\rangle$, and they have significant importance in the theory:
\begin{itemize}
  \item $\langle \mathcal I_{(0,-)}\rangle$ is the homogeneous defining ideal of the special fiber.
  \item $\langle \mathcal I_{(-,1)}\rangle$ coincides with the ideal of $T$  generated by the biforms $s_1y_1+\cdots+s_ny_n\in T$, whenever $(s_1,\ldots,s_n)$ is a syzygy of $g_1,\ldots,g_n$ of certain degree in $R$. Therefore, $T/\langle \mathcal I_{(-,1)}\rangle$ is a presentation of the symmetric algebra $\mathcal S (I)$ of $I$.
\end{itemize}

The ideal $I$ is said to be of {\em linear type} provided $\mathcal I=\langle \mathcal I_{(-,1)}\rangle$, and it is said to be of {\em fiber type} if $\mathcal I= \langle \mathcal I_{(-,1)}\rangle + \langle \mathcal I_{(0,-)}\rangle$.

\bigskip

Let $\A\subset\mathbb P^{k-1}$ be a rank $k$ arrangement of $s$ hyperplanes, defined by the linear forms $\ell_1,\ldots,\ell_s\in R:=\mathbb K[x_1,\ldots,x_k]$; in this case $\gcd(\ell_i,\ell_j)=1$ for all $i\neq j$. For $1\leq a\leq s$, consider the ideal $I:=I_a(\A)$. In terms of generators for $I$, we pick the standard generators $\ell_{i_1}\cdots\ell_{i_a}$, for all $1\leq i_1<\cdots<i_a\leq s$. Of course, this set of generators is not minimal. In fact, \cite[Proposition 2.10]{AnGaTo} gives the formula
$$\displaystyle \mu(I)=\sum_{u=0}^{\min\{k,s-a\}}c_{k-u,s-a-u},$$
where $c_{i,j}$ are coefficients occurring in the Tutte polynomial of the matroid of $\A$.

Denote $\mathcal I(\A,a)$ the presentation ideal of $R[It]$ for the above chosen set of generators.

\begin{rem} \label{gensRees} For any integer $e\geq 1$, we have $I^e=I_{ea}(\A(e))$, where $\A(e):=(\underbrace{\ell_1,\ldots,\ell_1}_{e},\ldots, \underbrace{\ell_s,\ldots,\ell_s}_{e})$. So, from Theorem \ref{theorem1}, we have that $I^e$ has linear graded free resolution. But this translates into $I$ having {\em linear powers}, which, by \cite[Theorem 2.5]{BrCoVa}, is equivalent to ${\rm reg}_{(1,0)}(R[It])=0$. In particular, what this means is that $\mathcal I(\A,a)$ doesn't have any minimal generators in bidegree $(u,v)$, with $u\geq 2$.
\end{rem}

There are several values of $a$ when we have some information about $\mathcal I(\A,a)$:
\begin{itemize}
  \item If we denote $\nu(\A)$ the maximal size of a coatom in the intersection latice of $\A$ (i.e., the maximum number of hyperplanes of $\A$ that intersect at a point), and if $1\leq a\leq s-\nu(\A)$, then $I=\langle x_1,\ldots,x_k\rangle^a$ (see \cite{To}), and the Rees algebra of $I$ is very well understood (see the discussions in \cite[Section 3.3]{To3}).
  \item The case $a=s-1$ has been treated extensively in \cite{GaSiTo}.
  \item When $a=s$, then $I$ is a principal ideal generated by $\ell_1\cdots\ell_s$.
\end{itemize}

This section is dedicated entirely to the next case which is $a=s-2$. Because a lot of calculations depend on the size of circuits in the matroid of $\A$, if we assume that the rank of $\A$ is 3, then any dependent set has size $\leq 4$ (i.e., any four of the defining linear forms are linearly dependent), and this eases up the computations quite a bit.

\subsection{The case $k=3$ and $a=s-2$.} Let $\mathcal A\subset\mathbb P^2$ be a line arrangement of rank 3, defined by linear forms $\ell_1,\ldots,\ell_s\in R:=\mathbb K[x,y,z]$. Consider $I:=I_{s-2}(\A)$, and denote $\mathcal I:=\mathcal I(\A,s-2)$.

\cite[Proposition 3.6]{To3} presents the generators for the symmetric ideal of $I$, i.e., ${\rm sym}(I):=\langle \mathcal I_{(-,1)}\rangle$. Also, in the same paper we started presenting the generators of the presentation ideal of the special fiber, i.e., $\partial(I):=\langle \mathcal I_{(0,-)}\rangle$. Next we review these results, and the notations.

Let $\displaystyle f_{i,j}:=\frac{\ell_1\cdots\ell_s}{\ell_i\ell_j}$, $1\leq i<j\leq s$ be the (standard) generators of $I$. By \cite[Theorem 2.4]{To3} we have $\displaystyle \mu(I)={{s}\choose{2}}-\sum_{j=1}^t{{{n_j}-1}\choose{2}}$, where $Sing(\A):=\{P_1,\ldots,P_t\}$ is the set of all intersection points of the lines of $\A$, and $n_j$ is the number of lines of $\A$ intersecting at the point $P_j$ (indeed, $n_j=\nu_{\A}(I(P_j))$).
The Rees ideal, $\mathcal I$, corresponding to this set of generators is the kernel of the map:
$$T:=R[\ldots, t_{i,j},\ldots]\longrightarrow R[It],\, t_{i,j}\mapsto f_{i,j}t.$$

\medskip

Recall \cite[Lemma 3.2 and Proposition 3.6]{To3} give the following information about important elements in $\mathcal I$.

\begin{itemize}
  \item[(I)] If $\{i_1,i_2,i_3\}, 1\leq i_1<i_2<i_3\leq s$, is a circuit in the matroid of $\mathcal A$, then it gives a dependency $c_{i_1}\ell_{i_1}+c_{i_2}\ell_{i_2}+c_{i_3}\ell_{i_3}=0$. This in turn gives the following element of $\mathcal I$:
      $$L_{i_1,i_2,i_3}:=c_{i_1}t_{i_2,i_3}+c_{i_2}t_{i_1,i_3}+c_{i_3}t_{i_1,i_2}.$$
  \item[(II)] For any $1\leq a<b<c\leq s$ the followings are elements of $\mathcal I$:
  \begin{eqnarray}
  A_{a,b,c}&:=&\ell_at_{a,b}-\ell_ct_{b,c},\nonumber\\
  B_{a,b,c}&:=&\ell_at_{a,c}-\ell_bt_{b,c},\nonumber\\
  C_{a,b,c}&:=&\ell_bt_{a,b}-\ell_ct_{a,c}.\nonumber
  \end{eqnarray}
  \item[(III)] If $s\geq 4$, for any $1\leq a<b<c<d\leq s$, the followings are elements of $\mathcal I$:
  \begin{eqnarray}
  Q^1_{a,b,c,d}:=t_{a,b}t_{c,d}-t_{a,c}t_{b,d},\nonumber\\
  Q^2_{a,b,c,d}:=t_{a,b}t_{c,d}-t_{a,d}t_{b,c}.\nonumber
  \end{eqnarray}
\end{itemize}
We have that ${\rm sym}(I)$ is generated by the sets (I) and (II). Also the sets (I) and (III) belong to $\partial(I)$.

There is another set of elements that belong to $\partial(I)$. Since the rank of $\A$ is 3, then any four of the defining linear forms are linearly dependent. By (I), we are left to analyze the circuits of size four; for example $\{1,2,3,4\}$, where any subset of three elements of this circuit is independent. This circuit comes with the linear dependency $d_1\ell_1+d_2\ell_2+d_3\ell_3+d_4\ell_4=0$, where all the coefficients are not zero. Now we follow the ideas in \cite{To3}, on how to obtain elements of $\partial(I)$ from the elements of the Orlik-Terao ideal. Our circuit of size four leads to the following element of the Orlik-Terao ideal
$$G:=d_1y_2y_3y_4+d_2y_1y_3y_4+d_3y_1y_2y_4+d_4y_1y_2y_3\in S:=\mathbb K[y_1,\ldots,y_s].$$ Multiplying this by any $y_k$, after pairing two $y$'s with different indices, via the preimage of the map in \cite[Proposition 3.3]{To3}, we obtain, modulo elements of type (III), the following elements of $\partial(I)$:
\begin{eqnarray}
P^1_{1,2,3,4}&:=&d_1t_{1,2}t_{3,4}+d_2t_{1,3}t_{1,4}+d_3t_{1,2}t_{1,4} +d_4t_{1,2}t_{1,3},\nonumber\\
P^2_{1,2,3,4}&:=&d_1t_{2,3}t_{2,4}+d_2t_{1,2}t_{3,4}+d_3t_{1,2}t_{2,4} +d_4t_{1,2}t_{2,3},\nonumber\\
P^3_{1,2,3,4}&:=&d_1t_{2,3}t_{3,4}+d_2t_{1,3}t_{3,4}+d_3t_{1,2}t_{3,4}+d_4t_{1,3}t_{2,3},\nonumber\\
P^4_{1,2,3,4}&:=&d_1t_{2,4}t_{3,4}+d_2t_{1,4}t_{3,4}+d_3t_{1,4}t_{2,4}+d_4t_{1,2}t_{3,4},\nonumber\\
R^k_{1,2,3,4}&:=&d_1t_{2,3}t_{4,k}+d_2t_{1,3}t_{4,k}+d_3t_{1,2}t_{4,k}+d_4t_{1,2}t_{3,k}, 5\leq k\leq s.\nonumber
\end{eqnarray}
Denote the set of all these elements with (IV), for all circuits $\{j_1,j_2,j_3,j_4\}, 1\leq j_1<j_2<j_3<j_4\leq s$.

\begin{rem} \label{remark1} In \cite[Subsection 3.2.1]{To3} it is shown how one can obtain canonically (i.e., via Sylvester forms) the elements of types (I) and (III), and some elements of type (IV), from elements of type (II). But after modulo elements of type (III), we also have $t_{2,3}P^1_{1,2,3,4}=t_{1,3}P^2_{1,2,3,4}$ and $t_{4,k}P^1_{1,2,3,4}=t_{1,4}R^k_{1,2,3,4}$. So via Sylvester forms we can obtain all elements of (IV) from elements of type (II).

In \cite[Example 3.5]{To3} it is obtained a minimal generator of $\partial(I)$ that is not of any of the types (I), (III), nor (IV): $$F:=t_{2,4}t_{3,4}+t_{1,4}t_{3,4}-t_{1,4}t_{2,4}.$$ That generator was obtained from the dependency $$1\cdot\ell_1+1\cdot\ell_2+(-1)\cdot\ell_3+0\cdot\ell_4=0.$$ This gave the element of the Orlik-Terao ideal $G:=y_2y_3y_4+y_1y_3y_4-y_1y_2y_4$ (observe that we will not simplify by $y_4$; we could simplify if we want because $\partial(I)$ is a prime ideal not containing the variables, and we would get the standard minimal generator of the Orlik-Terao ideal corresponding to the circuit $\{1,2,3\}$). Multiplying $G$ in order by the variables $y_1, y_2, y_3, y_4$, modulo (III) we obtain:
\begin{eqnarray}
P^1_{1,2,3,4}&=& t_{1,4}L_{1,2,3},\nonumber\\
P^2_{1,2,3,4}&=& t_{2,4}L_{1,2,3},\nonumber\\
P^3_{1,2,3,4}&=& t_{3,4}L_{1,2,3},\nonumber\\
P^4_{1,2,3,4}&=&F.\nonumber
\end{eqnarray}

From now on we can include the elements similar to $F$ into the type (IV) ones, by allowing the set $\{1,2,3,4\}$ to be dependent (not necessarily minimally dependent, i.e., a circuit). To sum up, below we show how via Sylvester forms we can obtain all the elements of types (I), (III), and (IV), from elements of type (II).

Suppose we have the dependency $a_1\ell_1+a_2\ell_2+a_3\ell_3+a_4\ell_4=0$, where $a_1,a_2,a_3\neq 0$, but $a_4$ may equal to zero, with $\ell_1,\ell_2,\ell_4$ linearly independent. Suppose also that $a_3=-1$. Consider the following elements of type (II):
$$A_{1,2,4}=\ell_1t_{1,2}-\ell_4t_{2,4}, B_{1,2,4}=\ell_1t_{1,4}-\ell_2t_{2,4}, A_{1,2,3}=\ell_1t_{1,2}-(a_1\ell_1+a_2\ell_2+a_4\ell_4)t_{2,3}.$$
We have
$$\left[\begin{array}{l} A_{1,2,4}\\ B_{1,2,4}\\A_{1,2,3}
\end{array}\right]=\left[
\begin{array}{ccc}
t_{1,2}&0&-t_{2,4}\\
t_{1,4}&-t_{2,4}&0\\
t_{1,2}-a_1t_{2,3}&-a_2t_{2,3}&-a_4t_{2,3}
\end{array}
\right]\cdot \left[\begin{array}{l} \ell_1\\ \ell_2\\\ell_4
\end{array}\right].$$
Taking the determinant of the $3\times 3$ content matrix we obtain
$$t_{2,4}[a_1t_{2,3}t_{2,4}+a_2t_{2,3}t_{1,4}-t_{1,2}t_{2,4}+a_4t_{1,2}t_{2,3}]=t_{2,4}[P^2_{1,2,3,4}-a_2Q^2_{1,2,3,4}]\in \mathcal I.$$ By primality, $P^2_{1,2,3,4}\in\partial(I)$. Furthermore, if $a_4=0$, then $P^2_{1,2,3,4}=t_{2,4}L_{1,2,3}$, and so $L_{1,2,3}\in\partial(I)$. Since any other element of type (IV) corresponding to the dependent set $\{1,2,3,4\}$ can be obtained from $P^2_{1,2,3,4}$, the conclusion follows.
\end{rem}

\medskip

As we observed above, for any $1\leq j_1<j_2<j_3<j_4\leq s$, the (not necessarily minimal) dependent set $\{j_1,j_2,j_3,j_4\}$ leads to the canonical construction of elements of type (I), (III), (IV) of $\partial(I)$. The question is if there are any other elements of $\partial(I)$ that cannot be generated by elements of type (I), (III), and (IV). We claim that there aren't any.

\begin{rem}\label{exp-restrictions} As it is explained in \cite{To3}, any generator of $\partial(I)$ is obtained by pairing variables $y$'s with different indices in $$M\cdot (d_{j_1}y_{j_2}y_{j_3}y_{j_4}+d_{j_2}y_{j_1}y_{j_3}y_{j_4}+d_{j_3}y_{j_1}y_{j_2}y_{j_4}+d_{j_4}y_{j_1}y_{j_2}y_{j_3})=:MG_{j_1,\ldots,j_4},$$ where $d_{j_1}\ell_{j_1}+d_{j_2}\ell_{j_2}+d_{j_3}\ell_{j_3}+d_{j_4}\ell_{j_4}=0$, and $M$ is a monomial in $S:=\mathbb K[y_1,\ldots,y_s]$ of odd degree.

We can suppose $j_i=i$ for $i=1,\ldots,4.$ Let $M=y_{1}^{m_1}\cdots y_{n}^{m_n}$, where $m_1+\cdots +m_n =m$. Since we need to pair $y_i$'s in $MG_{1,\ldots,4}$, we have a natural restriction about the degrees: $m_a \leq$ than the sum of the exponents of all the other variables in that term if the variable $y_a$ shows up in a term of $MG_{1,\ldots,4}$, consequently we have $$2m_l\leq m+1,~\mathrm{for}~l=1,\ldots,4 ~\mathrm{and}~2m_l\leq m+3,~\mathrm{for}~l=5,\ldots,n.$$
A monomial $M$ satisfying these inequalities will be said to satisfy the {\em exponents restrictions}; and the pairings of the variables $y_i$'s that pull back to an element of $\partial(I)$ will be called {\em valid pairings}.
\end{rem}

\medskip

\begin{thm} \label{fibertype} Using the above notations, $\partial(I)$ is generated by elements of types (I), (III), and (IV).
\end{thm}
\begin{proof} As discussed in Remark \ref{exp-restrictions}, we can suppose that any generator of $\partial(I)$ is obtained by pairing variables $y$'s with different indices in $$M\cdot (d_{1}y_{2}y_{3}y_{4}+d_{2}y_{1}y_{3}y_{4}+d_{3}y_{1}y_{2}y_{4}+d_{4}y_{1}y_{2}y_{3})=:MG_{1,2,3,4},$$ where $d_{1}\ell_{1}+d_{2}\ell_{2}+d_{3}\ell_{3}+d_{4}\ell_{4}=0$, and $M$ is a monomial in $\mathbb K[y_1,\ldots,y_s]$ of odd degree satisfying the exponent restrictions.

\medskip

The first idea is to reduce the worked case to $\deg(M)\leq 3$.

So suppose $\deg(M)=m\geq 5$. The goal is to show that for any valid pairings of $MG_{1,\ldots,4}$, there exist $i\neq j$ such that $M=y_iy_jM'$, where $M'=y_{1}^{m'_1}\cdots y_{n}^{m'_n}$ and $m'=\deg(M')=m-2$, and $M'$ satisfies the exponent restrictions: $$2m'_l\leq m'+1,~\mathrm{for}~l=1,\ldots,4 ~\mathrm{and}~2m'_l\leq m'+3,~\mathrm{for}~l=5,\ldots,n.$$ With this at hand, modulo elements of type (III), the valid parings of $MG_{1,\ldots,4}$ will ``transfer'' to some valid pairings of $M'G_{1,\ldots,4}$. But by induction, this can be generated by elements of types (I), (III), and (IV), and the pullback of the pairing $y_iy_j$ will be just the variable $t_{i,j}$.

\boxed{Case\, 1.} Assume $2m_l< m+1,~\mathrm{for}~\mathrm{all}~l=1,\ldots,4 ~\mathrm{and}~2m_l< m+3,~\mathrm{for}~\mathrm{all}~l=5,\ldots,s$. Since $m$ is odd, then for any choice of $i$ and $j$, $i\neq j$, we have $m'_i=m_i-1$ and $2m'_i\leq m'+1~\mathrm{and}~2m'_i\leq m'+3$, depending if $i\in \{1,\ldots,4\}$ or $i\in \{5,\ldots,n\}$ (the same for $m'_j=m_j-1$). So $M'$ satisfies the exponent restrictions:
$$
2m'_l\leq m'+1,~\mathrm{for}~l=1,\ldots,4 ~\mathrm{and}~2m'_l\leq m'+3,~\mathrm{for}~l=5,\ldots,n.
$$
This is saying that no matter what the valid pairings we chose for $MG_{1,\ldots,4}$, we can write $M=y_iy_jM'$ for some $i\neq j$, with $M'$ satisfying the exponents restrictions\footnote{If $M=y_k^m$, since $m\geq 5$ and $\deg(G_{1, 2, 3,4})=3$, any valid pairing of $MG_{1,2,3,4}$ must lead to a valid pairing in $y_k^{m-3}, m-3\geq 2$, which is impossible. So there are $i\neq j$ with $m_i\geq 1$ and $m_j\geq 1$.}, and

\boxed{Case\, 2.} Suppose there are $k,l\in \{1,\ldots,4\}$ such that $2m_k=m+1=2m_l$. Then $2m_k+2m_l=2(m+1)$. Impossible.

\boxed{Case\, 3.} Suppose there are $k,l\in \{5,\ldots,n\}$ such that $2m_k=m+3=2m_l$. Then $2m_k+2m_l=2(m+3)$. Impossible.

\boxed{Case\, 4.} Suppose there is $k\in\{1,\ldots,4\}$ such that $2m_k=m+1$, and there is $l\in\{5,\ldots,n\}$ such that $2m_l=m+3$. Then $2m_l+2m_k=2(m+2)$. Impossible.

\boxed{Case\, 5.} Suppose $2m_k=m+1$ ($k\in\{1,\ldots,4\}$) and $2m_r<m+1,$ for all $r\in \{1,\ldots,4\}\setminus\{k\}$, and $2m_l<m+3$ for all $l\in \{5,\ldots,n\}$. Then we can choose $i=k$ and  $j$ any, and $M'$ satisfies the exponents restrictions. Since $m\geq 5$, then $m_k\geq 3$. So in any valid pairings of $MG_{1, 2, 3,4}$, there will be at least a $y_k$ not paired with any of the $y$'s showing in the expansion of $G_{1,2,3,4}$. But this $y_i=y_k$ must pair with another $y_j$ from $M$.

\boxed{Case\, 6.} Suppose $2m_i<m+1$ for all $i\in\{1,\ldots,4\}$, $2m_l=m+3$ for some $l\in\{5,\ldots,n\}$, and $2m_r<m+3$, for all $r\in \{5,\ldots,n\}\setminus\{l\}$. Then we can choose $i$ any, and $j=l$, and therefore $M'$ satisfies the exponents restrictions. Same as in the previous case, since $m\geq 5$, then $m_l\geq 4$. So in any valid pairings of $MG_{1,\ldots,4}$, there will be at least a $y_l$ not paired with any of the $y$'s showing in the expansion of $G_{1,\ldots,4}$. But this $y_j=y_l$ must pair with another $y_i$ from $M$.

\medskip

Above we showed how recursively we can drop the degree of the monomial $M$ by 2, if $m\geq 5$. If $m=1$, then we recover elements of type (IV). So we need to focus on the case $m=3$. We divide this case in subcases:
	
\begin{enumerate}

\item ${y_i}^2y_j, i\neq j$.

\begin{enumerate}

\item $i\in\{1,\ldots,4\}$; suppose $i=1$. Then, by looking at the last three terms of $${y_1}^2y_j (d_{1}y_{2}y_{3}y_{4}+d_{2}y_{1}y_{3}y_{4}+d_{3}y_{1}y_{2}y_{4}+d_{4}y_{1}y_{2}y_{3}),$$ we have only one valid pairings possible in those terms. The pullback looks $$d_{1}A+d_{2}t_{1,3}t_{1,4}t_{1,j}+d_{3}t_{1,2}t_{1,4}t_{1,j}+d_{4}t_{1,2}t_{1,3}t_{1,j},$$ where we have various options for the pairings that give $A$. But modulo elements of type (III)\footnote{For example $A=t_{1,2}t_{1,3}t_{4,j}$, we use the fact that $t_{1,3}t_{4,j}\equiv t_{3,4}t_{1,j}$.}, we can write the above as $$t_{1,j}(d_{1}t_{1,2}t_{3,4}+d_{2}t_{1,3}t_{1,4}+d_{3}t_{1,2}t_{1,4}+d_{4}t_{1,2}t_{1,3})=t_{1,j}P_{1,2,3,4}^{1}.$$

\item $i\in\{5,\ldots,n\}$, and $j\in \{1,\ldots,4\}$; suppose $j=1$. Then, by looking at the last three terms of $${y_i}^2y_1 (d_{1}y_{2}y_{3}y_{4}+d_{2}y_{1}y_{3}y_{4}+d_{3}y_{1}y_{2}y_{4}+d_{4}y_{1}y_{2}y_{3}),$$ we have only one valid pairings possible in those terms. The pullback looks $$d_{1}A+d_{2}t_{1,3}t_{1,i}t_{4,i}+d_{3}t_{1,2}t_{1,i}t_{4,i}+d_{4}t_{1,2}t_{1,i}t_{3,i},$$ where we have various options for the pairings that give $A$. But modulo elements of type (III)\footnote{For example $A=t_{1,2}t_{3,i}t_{4,i}$, we use the facts that $t_{1,2}t_{4,i}\equiv t_{1,i}t_{2,4}$, and $t_{2,4}t_{3,i}\equiv t_{2,3}t_{4,i}$.}, we can write the above as
                 $$t_{1,i}(d_{1}t_{2,3}t_{4,i}+d_{2}t_{1,3}t_{4,i}+d_{3}t_{1,2}t_{4,i}+d_{4}t_{1,2}t_{3,i})=t_{1,i}R_{1,2,3,4}^{i}.$$

\bigskip

    Now, suppose $j\in \{5,\ldots,n\}$ and $i\neq j$. So, we have $${y_i}^2y_j (d_{1}y_{2}y_{3}y_{4}+d_{2}y_{1}y_{3}y_{4}+d_{3}y_{1}y_{2}y_{4}+d_{4}y_{1}y_{2}y_{3}).$$ Suppose that the pullback of some valid pairing is $$d_{1}t_{2,i}t_{3,i}t_{4,j}+d_{2}t_{1,i}t_{3,i}t_{4,j}+d_{3}t_{1,i}t_{2,i}t_{4,j}+d_{4}t_{1,i}t_{2,i}t_{3,j}.$$       But $t_{3,i}t_{4,j}=t_{3,4}t_{i,j}$, $t_{2,i}t_{4,j}=t_{2,4}t_{i,j}$ and $t_{2,i}t_{3,j}=t_{2,3}t_{i,j}$ modulo elements of type (III), then we can rewrite            $$t_{i,j}(d_{1}t_{2,i}t_{3,4}+d_{2}t_{1,i}t_{3,4}+d_{3}t_{1,i}t_{2,4}+d_{4}t_{1,i}t_{2,3}).$$
             Again $t_{2,i}t_{3,4}=t_{2,3}t_{4,i},$ $t_{1,i}t_{3,4}=t_{1,3}t_{4,i}$, $t_{1,i}t_{2,4}=t_{1,2}t_{4,i}$ and $t_{1,i}t_{2,3}=t_{1,2}t_{3,i}$ modulo elements of type (III), so we can rewrite
             $$t_{i,j}(d_{1}t_{2,3}t_{4,i}+d_{2}t_{1,3}t_{4,i}+d_{3}t_{1,2}t_{4,i}+d_{4}t_{1,2}t_{3,i})=t_{i,j}R_{1,2,3,4}^{i}.$$		
	\end{enumerate}

\item $y_iy_jy_k$. A similar computations modulo elements of type (III) leads to the pullback of any valid pairings of $y_iy_jy_kG_{1,2,3,4}$ rewritten as $t_{u,v}F$ with $F$ of type (IV). As an example, suppose $i=1,j=2,k=3$. We must do valid pairings in $${y_1}y_2y_3 (d_{1}y_{2}y_{3}y_{4}+d_{2}y_{1}y_{3}y_{4}+d_{3}y_{1}y_{2}y_{4}+d_{4}y_{1}y_{2}y_{3}).$$ The pullback of the last term can only be $d_4t_{1,2}t_{1,3}t_{2,3}$. For each of the other three terms, there are three possible valid pairings; for example the first term can be $d_1t_{1,4}t_{2,3}^2, d_1t_{1,3}t_{2,3}t_{2,4}$, or $d_1t_{1,2}t_{2,3}t_{3,4}$. Suppose we have
$$d_1t_{1,4}t_{2,3}^2+d_2t_{1,3}^2t_{2,4}+d_3t_{1,2}t_{1,3}t_{2,4}+d_4t_{1,2}t_{1,3}t_{2,3}.$$ As $t_{1,3}t_{2,4}\equiv t_{1,4}t_{2,3}$, modulo elements of type (III), we can rewrite
$$t_{2,3}(d_1t_{1,4}t_{2,3}+d_2t_{1,3}t_{1,4}+d_3t_{1,2}t_{1,4}+d_4t_{1,2}t_{1,3}).$$ Since $t_{1,4}t_{2,3}\equiv t_{1,2}t_{3,4}$, modulo elements of type (III), we obtain $$t_{2,3}P_{1,2,3,4}^1.$$
		
\item ${y_i}^3,~i\in\{5,\ldots,n\}$; suppose $i=5$. From the relation $d_1\ell_1+d_2\ell_2+d_3\ell_3+d_4\ell_4=0$, we have
$${y_5}^3 (d_{1}y_{2}y_{3}y_{4}+d_{2}y_{1}y_{3}y_{4}+d_{3}y_{1}y_{2}y_{4}+d_{4}y_{1}y_{2}y_{3}).
    $$
		Then using the remark, we have only one option: $$\mathcal{P}_1:=d_{1}t_{2,5}t_{3,5}t_{4,5}+d_{2}t_{1,5}t_{3,5}t_{4,5}+d_{3}t_{1,5}t_{2,5}t_{4,5}+d_{4}t_{1,5}t_{2,5}t_{3,5}.$$
		From the relation $f_1\ell_1+f_2\ell_2+f_3\ell_3+f_5\ell_5=0$ (we can assume $f_1=d_1$),  we can associate the element in Orlik-Terao ideal $$ d_{1}y_{2}y_{3}y_{5}+f_{2}y_{1}y_{3}y_{5}+f_{3}y_{1}y_{2}y_{5}+f_{5}y_{1}y_{2}y_{3}.$$
		Multiplying by $y_4{y_5}^2$, the pullback of a valid pairings is:
		$$\mathcal{P}_2:=d_{1}t_{2,5}t_{3,5}t_{4,5}+f_{2}t_{1,5}t_{3,5}t_{4,5}+f_{3}t_{1,5}t_{2,5}t_{4,5}+f_{5}t_{1,2}t_{3,5}t_{4,5}.$$
		Note that $\mathcal{P}_2$ is already treated in the case (1) (a). From the two relations above we have the third relation:  $e_2\ell_2+e_3\ell_3+e_4\ell_4+e_5\ell_5=0$ with $e_2=d_2-f_2,~e_3=d_3-f_3,~e_4=d_4$ and $e_5=-f_5$. This relation gives us another element in Orlik-Terao ideal $$ e_{2}y_{3}y_{4}y_{5}+e_{3}y_{2}y_{4}y_{5}+e_{4}y_{2}y_{3}y_{5}+e_{5}y_{2}y_{3}y_{4}.$$
		Multiplying by $y_1{y_5}^2$, the pullback of a valid pairing is:
		$$\mathcal{P}_3:=e_{2}t_{1,5}t_{3,5}t_{4,5}+e_{3}t_{1,5}t_{2,5}t_{4,5}+e_{4}t_{1,5}t_{2,5}t_{3,5}+e_{5}t_{1,2}t_{3,5}t_{4,5}.$$
	    Note that $\mathcal{P}_3$ is already treated in the case (1) (a). But we obviously have $$\mathcal{P}_1=\mathcal{P}_2+\mathcal{P}_3.$$
\end{enumerate}

So the case $m=3$ is also completely analysed.
\end{proof}

\smallskip

\begin{thm}\label{fiber-type}
	Using the above notations, the ideal $I:=I_{s-2}(\A)$ is of fiber type.
\end{thm}

\begin{proof}
Let $F(\ldots,t_{i,j}, \ldots)\in \mathcal{I}$ a generator of degree $d+1$ of the Rees ideal of $I$. If $d=0$, then by definition, $F$ is a linear form in $t_{i,j}$, with constant coefficients, so it is and element of $\partial(I)$ (i.e., a linear combination of elements of type (I)).

Suppose $d\geq 1$. By the Remark \ref{gensRees} we can suppose $$F(\ldots,t_{i,j}, \ldots)=\sum_{n_{i_1,i_2}+\cdots +n_{i_u,i_{u+1}}=d}L_{i_1,i_2,\ldots,i_u,i_{u+1}}t_{i_1,i_2}^{n_{i_1,i_2}}\cdots t_{i_u,i_{u+1}}^{n_{i_u,i_{u+1}}}, $$ with $L_{i_1,i_2,\ldots,i_u,i_{u+1}} \in \mathbb{K}[x,y,z]$ a linear form.

Since $F(\ldots,t_{i,j}, \ldots)\in \mathcal{I}$, we have $F(\ldots,f_{i,j}, \ldots)=0,$ that is, $F(\ldots,\frac{f}{\ell_i\ell_j}, \ldots)=0,$ with $f=\ell_1\cdots\ell_s$. Multiplying by $f^d$, we have $f^dF(\ldots,\frac{f}{\ell_i\ell_j}, \ldots)=F(\ldots,\frac{f}{\ell_i}\frac{f}{\ell_j}, \ldots)=0$. Then we can consider $F(\ldots,y_iy_j,\ldots) \in \mathcal{I}(\mathcal{A},s-1)\subset \mathbb{K}[x,y,z,y_1,\ldots,y_s]$ the presentation ideal of the Rees algebra of the ideal $I_{s-1}(\mathcal{A})=\langle \ell_2\cdots \ell_s, \ldots, \ell_1\cdots \ell_{s-1}\rangle.$

By \cite[Theorem 4.2]{GaSiTo}, we know that $I_{s-1}(\mathcal{A})$ is of fiber type, that is, $$\mathcal{I}(\mathcal{A},s-1)=\langle\mathrm{sym}(I_{s-1}(\mathcal{A})),\partial(I_{s-1}(\mathcal{A}))\rangle,$$ where $\partial(I_{s-1}(\mathcal{A}))=\langle\mathcal{I}(\mathcal{A},s-1)_{(0,-)}\rangle$ the Orlik-Terao ideal of $\mathcal{A}.$

By \cite[Lemma 3.1(b)]{GaSiTo} and \cite{OrTe}, sets of generators of these ideals are:
    $$\mathrm{sym}(I_{s-1}(\mathcal{A}))=\langle \ell_iy_i-\ell_{i+1}y_{i+1}~|~1\leq i \leq s-1\rangle,~\mathrm{and}$$ $$\partial(I_{s-1}(\mathcal{A}))=\langle G_{j_1,\ldots,j_4}~|~d_{j_1}\ell_{j_1}+d_{j_2}\ell_{j_2}+d_{j_3}\ell_{j_3}+d_{j_4}\ell_{j_4}=0\rangle,$$ where $G_{j_1,\ldots,j_4}$ is described in the Remark \ref{exp-restrictions}.
So, we can write \begin{eqnarray}
    F(\ldots,y_{i}y_{j}, \ldots)=\sum_{n_{i_1,i_2}+\cdots +n_{i_u,i_{u+1}}=d}L_{i_1,i_2,\ldots,i_u,i_{u+1}}(y_{i_1}y_{i_2})^{n_{i_1,i_2}}\cdots (y_{i_u}y_{i_{u+1}})^{n_{i_u,i_{u+1}}}\nonumber\\
    =\sum_{i=1}^{s-1}\mathcal{B}_i(\ell_iy_i-\ell_{i+1}y_{i+1}) + \sum\mathcal{C}_{j_1,\ldots,j_4}G_{j_1,\ldots,j_4}\hspace{3.4cm}
    \end{eqnarray}
    with ${B}_i,\mathcal{C}_{j_1,\ldots,j_4} \in \mathbb K[x,y,z,y_1,\ldots,y_s].$  About these polynomials we have the following properties:
\begin{enumerate}
    	\item[(*)] All the monomials showing up in expression (1) are constructed from pairings $y_iy_j$'s, and these pairings must be {\em valid pairings} (as we say in Remark \ref{exp-restrictions}) in order to pull back to variables $t_{i,j}$ in the expression of $F$. Below, we'll use ``$\,\widehat{\_}$\,'' to denote this pull back.
    	\item[(**)] As $F(\ldots,y_{i}y_j,\ldots)$ has degree $2d$ in variables $y_{i}$'s and  degree $1$ in variables $x,y,z$, then $\mathcal{B}_i$ must have degree $2d-1$ in variables $y_i$'s and  degree $1$ in variables $x,y,z$. About $\mathcal{C}_{j_1,\ldots,j_4}$, it must have degree $1$ in variables $x,y,z$ and  have either degree $2d-3$, or $2d-2$, in variables $y_i$'s.
    \end{enumerate}
    Here, and below, by ``monomial'' we will understand a monomial in variables $y_i$'s.

    So, using the statements above and the fact that $d\geq 1$, we can suppose that in each monomial in $\mathcal{B}_i$ and $\mathcal{C}_{j_1,\ldots,j_4}$ there is a variable $y_r$ for some $r \in\{1,\ldots,s\}$. If $d=1$, then the degree of $\mathcal{C}_{j_1,\ldots,j_4}$ in variables $y_i$'s can only be 0, so the pull back (i.e., pairing of $y_iy_j\leftrightarrow t_{i,j}$) of this term will give a combination of generators of type (I) with coefficients polynomials in variables $x,y,z$, so an element of $\partial(I)$.

    \medskip

    \noindent{\bf Claim:} In regard to $\mathcal{B}_i$, for each $i\in \{1,\ldots,s\}$,  we can suppose $r\neq i,i+1$.

    \smallskip

    {\em Proof of Claim.} Suppose that we have a monomial in $\mathcal{B}_i$ of the form $\mathcal{L}y_i^{m_i}y_{i+1}^{m_{i+1}}$ with $m_i+m_{i+1}=2d-1$ and $\mathcal{L}\in \mathbb{K}[x,y,z]_1$. If $m_{i+1}=0$ (or $m_{i}=0$), then we will have $\mathcal{L}y_i^{2d-1}(\ell_iy_i-\ell_{i+1}y_{i+1})$ showing up in expression (1), but this contradicts (*). If $m_i,m_{i+1}>0$ then we have $$\mathcal{L}y_i^{m_i}y_{i+1}^{m_{i+1}}(\ell_iy_i-\ell_{i+1}y_{i+1}).$$ For (*) to happen (i.e., valid pairings), in the first monomial above we need to have $m_i=d-1$ and $m_{i+1}=d$, and in the second monomial we need to have $m_i=d$ and $m_{i+1}=d-1$; an obvious contradiction. And Claim is shown.

\medskip

    So, from the Claim above, for each monomial $\mathcal{M}$ of each $\mathcal{B}_i$ there is $Q\in \mathbb{K}[x,y,z,y_1,\ldots,y_s]_{2d-2}$ such that $$\mathcal{M}=Qy_r(\ell_iy_i-\ell_{i+1}y_{i+1}) =Q(\ell_iy_iy_r-\ell_{i+1}y_{i+1}y_r).$$
    If $r>i+1$, then the pull back looks $\widehat{\mathcal{M}}=\widehat{Q}B_{i,i+1,r}$ and if $r<i$, then the pullback looks $\widehat{\mathcal{M}}=\widehat{Q}C_{r,i,i+1}$, where $B_{i,i+1,r}, C_{r,i,i+1}$ are elements of the type (II).

\medskip

    Now we analyse each monomial of $\mathcal{C}_{j_1,\ldots,j_4}$ only in the variables $y_i's$ (we disregard the ``coefficients'' which are linear forms in variables $x,y,z$ since they do not affect the pull back to variables $t_{i,j}$).

    Suppose $(j_1,\ldots,j_4) = (1,\ldots,4)$. Let $\mathcal{N}=y_{1}^{n_1}y_{2}^{n_2}y_{3}^{n_3}y_{4}^{n_4}y_{5}^{n_5}\ldots y_{s}^{n_s}$ be such a monomial. In this case we have only to discuss the case $n_1+\cdots+n_s=2d-3$. \footnote{If $\deg(G_{1,\ldots,4})=2$ (i.e., one of $d_1,d_2,d_3,d_4$ is 0), then $G_{1,\ldots,4}$ pulls back to an element of type (I), and also, it is not difficult to see that we must have valid pairings of the $y_i$'s in $\mathcal{N}$.} The condition (*) (or the exponent restrictions in Remark \ref{exp-restrictions}) applied to $\mathcal{N}G_{1,\ldots,4}$ leads to $2n_i\leq 2d-2$, that is, $n_i\leq d-1$ for each $i\in \{1,\ldots,s\}.$

    By symmetry, we only need to analyse two cases: if $n_1\leq n_2 \leq \cdots \leq n_s$ or $n_s\leq n_{s-1} \leq \cdots \leq n_1$.

    Suppose we are in the first case. Then we can organize the term $\mathcal{N}G_{1,\ldots,4}$ in the following way:
    \begin{eqnarray}\label{eq1}
    \hspace{1cm}\mathcal{N}G_{1,\ldots,4}=(y_{1}^{n_1}y_{2}^{n_2}y_{3}^{n_3}y_{4}^{n_4}y_{5}^{n_5}\cdots y_{s}^{n_s-1})y_{s}(d_{1}y_{2}y_{3}y_{4}+d_{2}y_{1}y_{3}y_{4}+d_{3}y_{1}y_{2}y_{4}+d_{4}y_{1}y_{2}y_{3}).
    \end{eqnarray}

    Note that we can pair all the variables in the element $y_{1}^{n_1}y_{2}^{n_2}y_{3}^{n_3}y_{4}^{n_4}y_{5}^{n_5}\ldots y_{s}^{n_s-1}$. For that we need to have the exponent restrictions $n_i \leq n_1 +\cdots \widehat{n_i} + \cdots + n_s -1$ for each $i\in \{1,\ldots,s-1\}$ and $n_s -1 \leq n_1  \cdots + n_{s-1}$. But by the hypothesis we already have it for $i\leq s-1$. If we suppose $n_s -1 > n_1 +\cdots + n_{s-1}$, then we have $n_s-1 > 2d-3 - n_s$, that is, $2n_s > 2d-2 $, equivalently $n_s > d-1,$ which it is a contradiction. So, the pull back of the expression (\ref{eq1}) looks $\widehat{\mathcal{N}}R^{s}_{1,2,3,4}$, with $R^{s}_{1,2,3,4}$ element of the type (IV).

    If we suppose $n_s\leq n_{s-1} \leq \cdots \leq n_1$ we can organize in the same way and the pull back of the expression (\ref{eq1}) looks $\widehat{\mathcal{N}}P^{1}_{1,2,3,4}$, with $P^{1}_{1,2,3,4}$ element of the type (IV).

    \medskip

    So, pulling back $F(\ldots,y_iy_j,\ldots)$ to $F(\ldots,t_{i,j},\ldots)$, we can see that it belongs to ${\rm sym}(I)+\partial(I)$, hence $I$ is of fiber type.
    \end{proof}

\bigskip

We end this section asking the two natural questions. In the case $a=s-2$ and $k=3$, is the Rees algebra of $I_{s-2}(\A)$ Cohen-Macaulay, or more specifically, is the special fiber (i.e., the Orlik-Terao algebra of the second order) Cohen-Macaulay? When $a=s-1$ and any $k$, these questions were answered affirmatively in \cite{GaSiTo}.

\bigskip

\section{Star Configurations}

\medskip

Let $R:=\mathbb K[x_0, \ldots, x_N]$ be a ring of homogeneous polynomials over a field $\mathbb K$. Let $\mathcal{A}=\{H_1,\ldots,H_s\}$ be a collection of $s\geq N+1$ distinct hyperplanes in $\mathbb P^N$, and suppose $\ell_1,\ldots,\ell_s\in R$ are defining linear forms of these hyperplanes: i.e., $H_i=V(\ell_i), i=1,\ldots,s$. We assume these hyperplanes meet {\it properly}, by which we mean that the intersection of any $j$ of these hyperplanes is either empty or has codimension $j$. For any $1\leq c\leq N$, let $V_c\left(\mathcal{A}\right)$ be the union of the codimension $c$ linear varieties defined by all the intersections of these hyperplanes, taken $c$ at a time:
$$
V_c\left(\mathcal{A}\right):=\bigcup_{1\leq j_1<\cdots<j_c\leq s}H_{j_1}\cap\cdots\cap H_{j_c}.
$$
We call $V_c\left(\mathcal{A}\right)$ {\em the star configuration of codimension $c$ with support $\mathcal{A}$ (in $\mathbb P^N$)}.
It is clear that the defining ideal is
$$I(V_c(\mathcal{A}))=\bigcap_{1\leq j_1<\cdots<j_c\leq s}\langle \ell_{j_1},\ldots,\ell_{j_c}\rangle,$$
and the {\em $m$-th symbolic power} of this ideal is $$I(V_c(\mathcal{A}))^{(m)}=\bigcap_{1\leq j_1<\cdots<j_c\leq s}\langle \ell_{j_1},\ldots,\ell_{j_c}\rangle^m.$$

For any homogeneous ideal $J$ in $R$, we use $\alpha(J)$ to denote the least degree of a nonzero element in $J$. By \cite[Corollary 4.6]{GeHaMi}, one has $\alpha\left(I(V_c(\mathcal{A}))^{(m)}\right)=(q+1)s-c+r$ if $m=qc+r$ for $1\leq r\leq c$.

We will prove that the defining ideal of  star configuration of codimension $c$ satisfies the following properties:

\begin{prop} \label{star}
Let $I=I(V_c(\mathcal{A}))$ be the defining ideal of a star configuration of codimension $c$ and $M=\langle x_0, \ldots, x_N\rangle$  the irrelevant maximal ideal of  $R$. Then for $m\geq 1$, $I$ satisfies the following properties:
\begin{enumerate}
\item $I^{(mc)}\subseteq M^{m(c-1)}I^m$. In particular, if \, ${\rm ht}(I)=N$, then $I^{(mN)}\subseteq M^{m(N-1)}I^m$.
\item $I^{(mc-c+1)}\subseteq I^m$ and  $I^{(mN-N+1)}\subseteq I^m$.
\item $I^{(mc-c+1)}\subseteq M^{(m-1)(c-1)}I^m$. In particular, if \, ${\rm ht}(I)=N$, then $I^{(mN-N+1)}\subseteq M^{(m-1)(N-1)}I^m$.
\item For all $t\geq 0$, one has $\frac{\alpha\left(I^{(t)}\right)}{t}\geq \frac{\alpha\left(I^{(m)}\right)+c-1}{m+c-1}$.
\item For all $t\geq 0$, one has $\frac{\alpha\left(I^{(t)}\right)}{t}\geq \frac{\alpha\left(I^{(m)}\right)+N-1}{m+N-1}$.
\item For all $t\geq 0$, one has $\frac{\alpha\left(I^{(t)}\right)}{t}\geq \frac{\alpha\left(I\right)+N-1}{N}$.
\end{enumerate}
\end{prop}

\begin{proof} (1)  By \cite[Theorem A]{EiLaSm} or \cite[Theorem 1.1]{HoHu}, we have $I^{(mc)}\subseteq I^m$.
By \cite[Corollary 4.6]{GeHaMi}, we have $\alpha\left(I^{(mc)}\right)=ms$. Since $I$ is generated in degree $s-c+1$, we have $I^m$ is generated in degree $m(s-c+1)$. Since $ms=m(s-c+1)+m(c-1)$, we have $I^{(mc)}= [I^{(mc)}]_{t\geq ms}\subseteq[I^{m}]_{t\geq ms}=[M^{m(c-1)}I^m]_{t\geq ms}$. So
$I^{(mc)}\subseteq M^{m(c-1)}I^m$ for all $m\geq 1$.

(2) By \cite[Theorem 3.2]{ToXi} or by Theorem \ref{theorem1}, for $m\geq 1$, one has
$$
I^m=I(V_c(\mathcal{A}))^m=I(V_c(\mathcal{A}))^{(m)}\cap I(V_{c+1}(\mathcal{A}))^{(2m)}\cap\cdots\cap I(V_N(\mathcal{A}))^{((N-c+1)m)}\cap M^{(s-c+1)m}.
$$
Since $mc-c+1=(m-1)c+1$, by \cite[Corollary 4.6]{GeHaMi}, we have $\alpha\left(I^{(mc-c+1)}\right)=ms-c+1=m(s-c+1)+(m-1)(c-1)\geq m(s-c+1)$. Hence $I^{(mc-c+1)}\subseteq M^{(s-c+1)m}$.

Now for any $c\leq i\leq N$ and any codimension $i$ linear variety defined by $\mathcal{A}^{\prime}=\{H_{j_1}, \ldots, \mathcal{H}_{j_i}\}$,  we have
$$
I(V_c(\mathcal{A}))^{(mc-c+1)}\subseteq I(V_c(\mathcal{A}^{\prime}))^{(mc-c+1)}.
$$
We can treat $V_c(\mathcal{A}^{\prime})$ as a star configuration of codimension $c$ in $\mathbb{P}^{i-1}$. Again by \cite[Corollary 4.6]{GeHaMi}, we have $\alpha\left(I(V_c(\mathcal{A}^{\prime}))^{(mc-c+1)}\right)=mi-c+1=m(i-c+1)+(m-1)(c-1)\geq m(i-c+1)$. Hence
$$
I(V_c(\mathcal{A}))^{(mc-c+1)}\subseteq I(V_c(\mathcal{A}^{\prime}))^{(mc-c+1)}\subseteq \langle\ell_{j_1}, \ldots, \ell_{j_i}\rangle^{m(i-c+1)}.
$$
As $\mathcal{A}^{\prime}$ ranges over all codimension $i$ linear varieties, one has $I(V_c(\mathcal{A}))^{(mc-c+1)}\subseteq I(V_{i}(\mathcal{A}))^{(m(i-c+1))}$ for any $c\leq i\leq N$. This completes the proof that $I^{(mc-c+1)}\subseteq I^m$.

The second inclusion $I^{(mN-N+1)}\subseteq I^m$ holds because $I^{(mN-N+1)}\subseteq I^{(mc-c+1)}$.

(3) By (2), we have $I^{(mc-c+1)}\subseteq I^m$.  Since $\alpha\left(I^{(mc-c+1)}\right)=ms-c+1=m(s-c+1)+(m-1)(c-1)$ and $I^m$ is generated in degree $m(s-c+1)$, we have $I^{(mc-c+1)}\subseteq M^{(m-1)(c-1)}I^m$.

(4) Write $t=q_1c+r_1$ for $1\leq r_1\leq c$ and  $m=q_2c+r_2$ for $1\leq r_2\leq c$. Then by \cite[Corollary 4.6]{GeHaMi}, we have $\alpha\left(I^{(t)}\right)=(q_1+1)s-c+r_1=t+(q_1+1)(s-c)$ and $\alpha\left(I^{(m)}\right)=(q_2+1)s-c+r_2=m+(q_2+1)(s-c)$. We need to show the following inequality
$$
\frac{t+(q_1+1)(s-c)}{t}\geq \frac{m+(q_2+1)(s-c)+c-1}{m+c-1}.
$$
which is equivalent to
$$
(q_1+1)(s-c)(m+c-1)\geq t(q_2+1)(s-c).
$$
The above inequality  holds because of the following argument:
$$
(q_1+1)(s-c)(m+c-1)- t(q_2+1)(s-c)
$$
$$
=(s-c)\left[(q_1+1)(m+c-1)-t(q_2+1)\right]
$$
$$
=(s-c)\left[(q_1+1)(q_2c+r_2+c-1)-(q_1c+r_1)(q_2+1)\right]
$$
$$
=(s-c)\left[(q_1+1)((q_2+1)c+r_2-1)-((q_1+1)c+r_1-c)(q_2+1)\right]
$$
$$
=(s-c)\left[(q_1+1)(r_2-1)+(c-r_1)(q_2+1)\right]\geq 0
$$
because $s\geq c$, $r_2\geq 1$, and $r_1\leq c$.

(5) By (4), we just need to show  $\frac{\alpha\left(I^{(m)}\right)+c-1}{m+c-1}\geq \frac{\alpha\left(I^{(m)}\right)+N-1}{m+N-1}$ for $1\leq c\leq N$. We first show $\alpha\left(I^{(m)}\right)\geq m$. Write $m=qc+r$ for $1\leq r\leq c$. Then by \cite[Corollary 4.6]{GeHaMi}, we have $\alpha\left(I^{(m)}\right)=(q+1)s-c+r$. Hence
$$
\alpha\left(I^{(m)}\right)-m=(q+1)s-c+r-qc-r=(q+1)(s-c)\geq 0
$$
which implies $\alpha\left(I^{(m)}\right)\geq m$. Now to show
$$
\frac{\alpha\left(I^{(m)}\right)+c-1}{m+c-1}\geq \frac{\alpha\left(I^{(m)}\right)+N-1}{m+N-1}=\frac{\alpha\left(I^{(m)}\right)+c-1+(N-c)}{m+c-1+(N-c)},
$$
we just need to show
$$
\left(\alpha\left(I^{(m)}\right)+c-1\right)(N-c)\geq (m+c-1)(N-c).
$$
But this follows because
$$
\left(\alpha\left(I^{(m)}\right)+c-1\right)(N-c)- (m+c-1)(N-c)=(N-c)\left(\alpha\left(I^{(m)}\right)-m\right)\geq 0.
$$

Finally (6) follows from (5) by setting $m=1$.
\end{proof}

\begin{rem} Proposition \ref{star} proves several conjectures of symbolic powers for the defining ideals of star configurations of codimension $c$.  Proposition \ref{star} (1) and (3) prove that these ideals satisfy a general version of  conjectures regarding to the containment of symbolic powers and regular powers  proposed by Harbourne and Huneke (see \cite[Conjecture 2.1, Conjecture 4.1.5]{HaHu}). Proposition \ref{star} (2) verifies them for  conjectures  proposed by Harbourne (see \cite[Conjecture 8.4.2, Conjecture 8.4.3]{BRHKKSS09}). Proposition \ref{star} (4) and (5) show that Demailly's conjecture (see \cite[page 101]{De} or \cite[Question 4.2.1]{HaHu}) holds true for these ideals. As a special case of Demailly's conjecture, Proposition \ref{star} (6) concludes  that Chudnovsky's conjecture (see \cite{Ch}) also holds for such ideals.

All of these conjectures were verified for star configurations of codimension $N$ (see \cite[Example 8.4.8]{BRHKKSS09}, \cite[Corollary 3.9, Corollary 4.1.7, and the comments after Question 4.2.1]{HaHu}).
\end{rem}

\vskip 0.3in

\noindent {\bf Acknowledgment.} Ricardo Burity is grateful to CAPES for funding his one year stay at University of Idaho.

\bigskip

\renewcommand{\baselinestretch}{1.0}
\small\normalsize 

\bibliographystyle{amsalpha}

\end{document}